\documentclass[10pt,a4paper]{amsart}
\usepackage{amsmath, amssymb, amsfonts, amsthm, float, stmaryrd, epsfig}
\usepackage[pdfborder={0 0 0 [0 0 ]},bookmarksdepth=3]{hyperref}
\usepackage[latin1]{inputenc}
\usepackage[capitalize]{cleveref}
\usepackage{color}
\usepackage[ps,all,arc,rotate]{xy}
\usepackage{bbm} 

\hypersetup{
    colorlinks=true,
    linkcolor=black,
    citecolor=black,
    filecolor=black,
    urlcolor=black,
}

\numberwithin{figure}{section}
\numberwithin{table}{section}

\theoremstyle{plain}
\newtheorem{thm}{Theorem}[section]
\crefname{thm}{Theorem}{Theorems}
\newtheorem*{prop*}{Proposition}
\newtheorem*{thm*}{Theorem}
\newtheorem{prop}[thm]{Proposition}
\crefname{prop}{Proposition}{Propositions}
\newtheorem{lem}[thm]{Lemma}
\crefname{lem}{Lemma}{Lemmata}
\newtheorem{cor}[thm]{Corollary}
\crefname{cor}{Corollary}{Corollaries}

\theoremstyle{definition}

\newtheorem{dfn}[thm]{Definition}
\newtheorem*{dfn*}{Definition}

\theoremstyle{remark}
\newtheorem{rmk}[thm]{Remark}

\newtheoremstyle{maintheorem}{}{}{\itshape}{}{\bfseries}{}{.5em}{#1 \!\thmnote{#3}.}
\theoremstyle{maintheorem}
\newtheorem*{mainthm}{Theorem}

\makeatletter
\let\c@figure\c@thm
\let\c@table\c@thm
\makeatother
\crefname{figure}{Figure}{Figures}
\crefname{table}{Table}{Tables}

\newcommand{\Aut}{\operatorname{Aut}}
\newcommand{\SAut}{\operatorname{SAut}}
\newcommand{\Out}{\operatorname{Out}}

\newcommand{\GL}{\operatorname{GL}}
\newcommand{\SL}{\operatorname{SL}}

\newcommand{\PGL}{\operatorname{PGL}}

\renewcommand{\sp}{\operatorname{Sp}}

\newcommand{\im}{\operatorname{im}}

\newcommand{\Alt}{\operatorname{Alt}\nolimits}
\newcommand{\Sym}{\operatorname{Sym}\nolimits}

\newcommand{\I}{\mathrm{I}}

\newcommand{\A}{\operatorname{\mathtt{A}}}
\newcommand{\Atwo}{\operatorname{{^2\hspace{-1 pt}}\mathtt{A}}}
\newcommand{\B}{\operatorname{\mathtt{B}}}
\newcommand{\Btwo}{\operatorname{{^2}\mathtt{B}}}
\newcommand{\typeC}{\operatorname{\mathtt{C}}}
\newcommand{\D}{\operatorname{\mathtt{D}}}
\newcommand{\Dtwo}{\operatorname{{^2}\mathtt{D}}}
\newcommand{\Dthree}{\operatorname{{^3}\mathtt{D}}}
\newcommand{\E}{\operatorname{\mathtt{E}}}
\newcommand{\Etwo}{\operatorname{{^2}\mathtt{E}}}
\newcommand{\typeF}{\operatorname{\mathtt{F}}}
\newcommand{\typeFtwo}{\operatorname{{^2}\mathtt{F}}}
\newcommand{\G}{\operatorname{\mathtt{G}}}
\newcommand{\Gtwo}{\operatorname{{^2}\mathtt{G}}}

\newcommand{\Fi}{\operatorname{Fi}}

\def\C{\mathbb{C}}

\def\Z{\mathbb{Z}}

\def\1{\mathbbm{1}}

\def\F{\mathbb{F}}

\def\s-{\smallsetminus}

\newcommand{\mcgop}{\operatorname{MCG}}
\newcommand{\mcgcl}[1]{\mcgop(\Sigma_{{#1}})}
\newcommand{\mcg}[2]{\mcgop(\Sigma_{{#1},{#2}})}

\newcounter{dawidcomments}

\newcounter{emiliocomments}

\newcounter{barbaracomments}

\author{Dawid Kielak and Emilio Pierro}
\title{On the smallest non-trivial quotients of mapping class groups}
\date{\today}

\begin{document}

\begin{abstract}
We prove that the smallest non-trivial quotient of the mapping class group of a connected orientable surface of genus $g \geqslant 3$ without punctures is $\sp_{2g}(2)$, thus confirming a conjecture of Zimmermann. In the process, we generalise Korkmaz's results on $\C$-linear representations of mapping class groups to projective representations over any field.
\end{abstract}

\maketitle
\section{Introduction}

In \cite{Zimmermann2012} Zimmermann conjectured that the smallest quotient of the mapping class group of a surface  of genus $g$ (with $g \geqslant 3$) is $\sp_{2g}(2)$, the symplectic group of rank $g$ over the field of two elements. Zimmermann proved this statement for $g \in \{3,4\}$. We confirm his conjecture in general, and prove
\begin{mainthm}[\ref{main thm}]
 Let $K$ be a non-trivial finite quotient of $\mcg{g}{b}$, the mapping class group of a connected orientable surface $\Sigma_{g,b}$ of genus $g \geqslant 3$ with $b$ boundary components. Then either $\vert K \vert > \vert \sp_{2g}(2) \vert$, or $K \cong \sp_{2g}(2)$ and the quotient map is obtained by postcomposing the natural map $\mcg{g}{b} \to \sp_{2g}(2)$ with an automorphism of $\sp_{2g}(2)$.
\end{mainthm}

The natural map is obtained as a sequence of surjections: we start by observing that $\mcg{g}{b}$ surjects by \cite[Proposition 3.19 and Theorem 4.6]{FarbMargalit2012} onto $\mcg{g}{0}$ (which we will denote by $\mcgcl{g}$) -- the surjection is obtained by `capping', that is gluing discs to the boundary components of $\Sigma_{g,b}$. The group $\mcgcl{g}$ acts on
\[\mathrm{H}_1(\Sigma_{g}) \cong \Z^{2g}\]
in a way preserving the algebraic intersection number, which is a symplectic form. This leads to a homomorphism
\[
 \mcgcl{g} \to \sp_{2g}(\Z)
\]
which is surjective by \cite[Theorem 6.4]{FarbMargalit2012}.
Reducing integers modulo $2$ we obtain an epimorphism $\sp_{2g}(\Z) \to \sp_{2g}(2)$, and this way we obtain the natural map $\mcg{g}{b} \to \sp_{2g}(2)$.

Note that $\mcgcl{g}$ has plenty of finite quotients -- Grossman showed in~\cite{Grossman1974} that it is residually finite. (In the same paper Grossman showed that $\Out(F_n)$ is residually finite as well.)

The situation changes when we allow punctures: the mapping class group of a surface with $n$ punctures maps onto the symmetric group $\Sym_n$, and such a quotient can be smaller than $\sp_{2g}(2)$. If we however look at a pure mapping class group, our theorem applies, since such a mapping class group can be obtained from $\mcg{g}{b}$ by capping the boundary components with punctured discs; the resulting homomorphism is surjective.

When $\Gamma$ is a closed non-orientable surface of genus $g \geqslant 3$, then it contains a subsurface homeomorphic to $\Sigma_{g,1}$, and so \cref{main thm} implies that any finite quotient of $\mcgop(\Gamma)$ is either of cardinality at most $2$, or at least $\vert \sp_{2g}(2) \vert$ (we are using the fact that Dehn twists along simple non-separating two-sided curves generate an index $2$ subgroup of $\mcgop(\Gamma)$, as shown by Lickorish~\cite{Lickorish1965}).

\smallskip

Our result coheres nicely  with the theorem of Berrick--Gebhardt--Paris~\cite{Berricketal2014} which states that the smallest non-trivial  action of $\mcg{g}{b}$ on a set is unique and obtained by mapping $\mcg{g}{b}$ onto $\sp_{2g}(2)$ and taking the smallest permutation representation of the latter group.
Let us remark here that the discussion above has a parallel in the setting of automorphisms of free groups: Baumeister together with the authors has shown in \cite{Baumeisteretal2017} that the smallest non-abelian quotient of $\Aut(F_n)$ is $\SL_n(2)$; in the same paper it is shown that (for large $n$) every action of the group $\SAut(F_n)$ of pure automorphisms of $F_n$ on a set of cardinality $2^{n/2}$ is trivial, and the smallest known non-trivial action comes from the smallest action of $\SL_n(2)$, and is defined on a set of cardinality $2^n-1$.

\smallskip

The proof of \cref{main thm} follows the same general outline as in \cite{Baumeisteretal2017}: since $\mcg{g}{b}$ is perfect, we know that its smallest non-trivial quotient is simple and non-abelian. We go through the Classification of Finite Simple Groups (CFSG) and exclude all such groups smaller than $\sp_{2g}(2)$ from the list of potential quotients.
The finite simple groups fall into one of the following four families:\begin{enumerate}
\item the cyclic groups of prime order;
\item the alternating groups $\Alt_n$, for $n \geqslant 5$;
\item the finite groups of Lie type; and
\item the $26$ sporadic groups.
\end{enumerate}
For the full statement of the CFSG we refer the reader to \cite{raw}. For the purpose of this paper, we further divide the finite groups of Lie type into the following two families:
\begin{enumerate}
\item[($3$C)] the ``classical groups'': $\A_n$, $\Atwo_n$, $\B_n$, $\typeC_n$, $\D_n$ and $\Dtwo_n$; and
\item[($3$E)] the ``exceptional groups'': $\Btwo_2$, $\Gtwo_2$, $\typeFtwo_4$, $\Dthree_4$, $\Etwo_6$, $\G_2$, $\typeF_4$, $\E_6$, $\E_7$ and $\E_8$.
\end{enumerate}

The cyclic groups are excluded since they are abelian. The alternating groups are easily dealt with using the aforementioned result of Berrick--Gebhardt--Paris (see \cref{bgparis} and the corollary following it).

To deal with the classical groups we investigate the low-dimensional representation theory of $\mcg{g}{b}$. Korkmaz in \cite{Korkmaz2011} (see also \cite{FranksHandel2013} by Franks--Handel) showed that every homomorphism $\mcg{g}{b} \to \GL_n(\C)$ has abelian image whenever $g \geqslant 1$ and $n<2g$ (in particular, the image is trivial when $g \geqslant 3$). His method can actually be used to obtain the more general
\begin{mainthm}[\ref{main thm reps}]
Every projective representation of $\mcg{g}{b}$ with $g\geqslant 3$ and $b \geqslant 0$ in dimension less than $2g$ is trivial.
\end{mainthm}

Note that the representation theory of mapping class groups is in general poorly understood, even when compared with the representation theory of $\Out(F_n)$: it is still an open problem whether mapping class groups are linear; in this context let us mention a result of Button \cite{Button2016}, who proved that mapping class groups are not linear over fields of positive characteristic. The groups $\Out(F_n)$ are not linear over any field, as shown by Formanek--Procesi~\cite{FormanekProcesi1992}.
In the setting of mapping class groups there is also no satisfactory counterpart to the result of the first-named author, who in \cite{Kielak2013} classified linear representations of $\Out(F_n)$ (for $n \geqslant 6$ and over fields of characteristic not dividing $n+1$) in all dimensions below $\binom {n+1} 2$. The best result in this direction is proven by Korkmaz in~\cite{Korkmaz2011a}, where he classified all representations of $\mcg{g}{b}$ over $\C$ up to dimension $2g$.

\subsection*{Acknowledgements.} The authors are grateful to Barbara Baumeister and Stefan Witzel for helpful conversations.

The first-named author was supported by the DFG Grant KI-1853.
The second-named author was supported by the SFB 701 of Bielefeld University.

\section{Mapping class groups and their representations}

\begin{figure}
\begin{center}
 \includegraphics[scale=0.4]{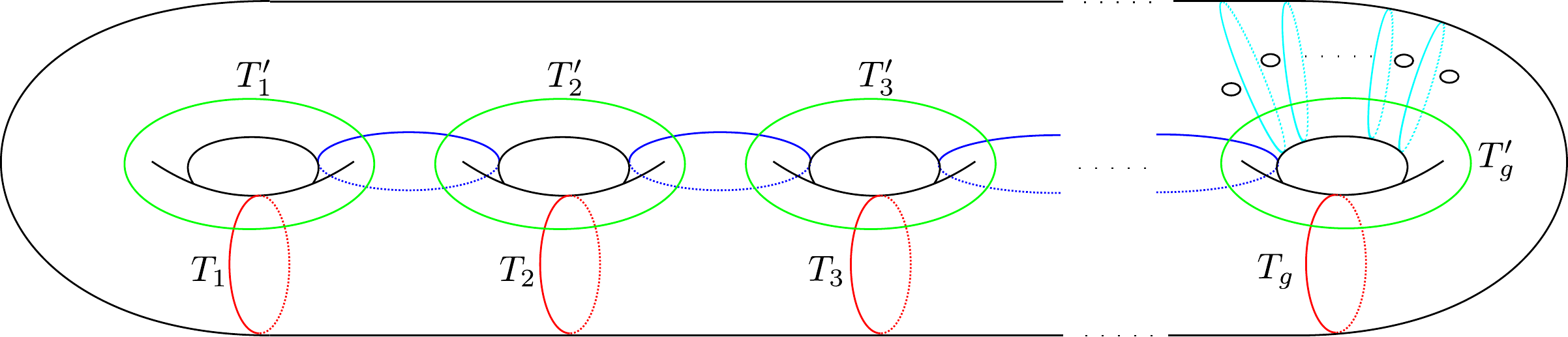}
 \caption{}
 \label{surface}
 \end{center}
 \end{figure}

 Throughout, we take $\Sigma_{g,b}$ to be a connected orientable surface of genus $g$ with $b$ boundary components. 
We allow boundary components as they accommodate inductive arguments; we do not allow punctures.

We let $\mcg{g}{b}$ denote the mapping class group of $\Sigma_{g,b}$, that is the group of isotopy classes of orientation preserving homeomorphisms of $\Sigma_{g,b}$ which preserve each boundary component pointwise. The isotopies are also required to fix the boundary components pointwise.

\cref{surface} depicts a choice of $3g+b-2$ simple closed curves on $\Sigma_{g,b}$ (in fact $3g-1$ when $b=0$); any two of these curves can be permuted by homeomorphisms of $\Sigma_{g,b}$. In fact more is true: using the classification of surfaces, we immediately see that any two non-separating simple closed curves in $\Sigma_{g,b}$ can be mapped to one another using a homeomorphism of $\Sigma_{g,b}$.

We let $\mathcal S$ denote the set of Dehn twists along the curves visible in \cref{surface}. The set $\mathcal S$ is a generating set
for $\mcg{g}{b}$ when $g \geqslant 2$; this and similar facts can easily be found in the book of Farb and Margalit~\cite{FarbMargalit2012}.
We let $T_1, \dots, T_g$ and $T_1', \dots, T_g'$ denote Dehn twists along the indicated curves.
When arguing about representations of $\mcg{g}{b}$, we will repeatedly use the braid relations: for every $i$ we have
\[
 T_i T_i' T_i = T_i'T_i T_i'
\]

Let us record the following classical result -- for a discussion on the history of the result, see \cite[Section 5.1]{FarbMargalit2012}. The formulation we state here is as given by Korkmaz.
\begin{thm}[{\cite[Theorem 5.1]{Korkmaz2002}} and {\cite[Theorem 3.4]{Korkmaz2011}}]
\label{perfect}
Let $b \geqslant 0$ and $g \geqslant 3$. The abelianisation of $\mcg{g}{b}$ is trivial. When $g=2$, the
abelianisation of $\mcg{2}{b}$ is isomorphic to the cyclic group of order $10$, and the derived subgroup of $\mcg{2}{b}$ is perfect.
\end{thm}

We will now use the method developed by Korkmaz~\cite{Korkmaz2011} to investigate low-dimensional representations of $\mcg{g}{b}$. Korkmaz originally studied $\C$-linear representations, whereas we are interested in projective representations over more general fields.

\begin{prop}
\label{killing dehn twists}
Let $g\geqslant 2$ and $b \geqslant 0$, and let $\phi \colon \mcg{g}{b} \to G$ be a homomorphism. Let $T$ and $S$ denote two Dehn twists along simple closed non-separating curves intersecting each other in a single point. The following are equivalent:
\begin{enumerate}
 \item $\phi(T) = \phi(S)$.
 \item $\phi(T)$ commutes with $\phi(S)$.
 \item $\phi$ factors through the abelianisation of $\mcg{g}{b}$.
\end{enumerate}
\end{prop}
\begin{proof}
\noindent \textbf{(1) $\Rightarrow$ (2)} This is clear.
\smallskip

\noindent \textbf{(2) $\Rightarrow$ (3)}
Note that $\mcg{g}{b}$ is generated by $\mathcal S$, and we immediately see that $\phi(T')$ commutes with $\phi(S')$ for every $T',S' \in \mathcal S$: either this is already true in $\mcg{g}{b}$, or we can find a homeomorphism $h$ of $\Sigma_{g,b}$ such that ${T'}^h=T$ and ${S'}^h = S$, and then use the fact that $\phi(T)$ and $\phi(S)$ commute. Therefore, $\phi$ takes any two elements of the generating set $\mathcal S$ to commuting elements in $G$, and so $\im \phi$ is abelian.

\smallskip
\noindent \textbf{(3) $\Rightarrow$ (1)}
All Dehn twists along simple closed non-separating curves are conjugate in $\mcg{g}{b}$, and hence so are $S$ and $T$. If $\phi$ has abelian image then $\phi(T) = \phi(S)$.
\end{proof}


Throughout, we use $C_G(x)$ to denote the centraliser of $x$ in $G$ (where $x \in G$).

\begin{figure}
\begin{center}
 \includegraphics[scale=0.4]{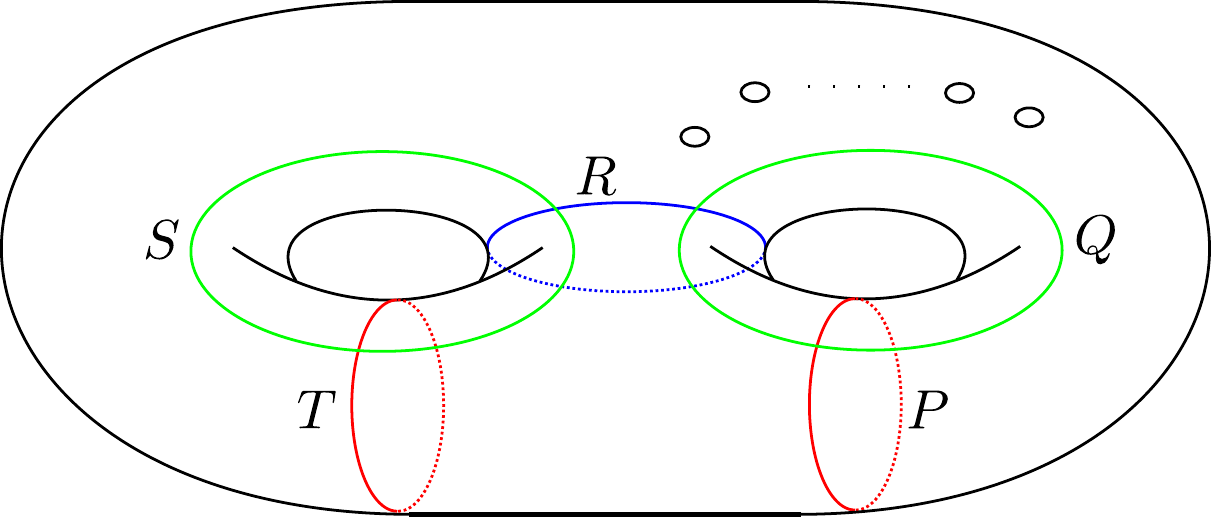}
 \caption{}
 \label{surface2}
 \end{center}
 \end{figure}

\begin{prop}
\label{reps base case}
Let $b \geqslant 0$.
Every projective representation \[\phi \colon \mcg{2}{b} \to \PGL_n(\F)\]
with $n\leqslant 3$ has abelian image.
\end{prop}
\begin{proof}
 We will proceed identically to \cite[Proposition 4.3]{Korkmaz2011}. By extending scalars we may assume that $\F$ is algebraically closed. This allows us to use the Jordan normal form.

 Because we are working with projective representations, we need to look at the cases $n=2$ and $n=3$ separately (the case $n=1$ is trivial, since $\PGL_1(\F)$ is the trivial group).

 Throughout, we will use Dehn twists $P,Q,R,S$, and $T$ supported on the curves depicted in \cref{surface2}. For any of these Dehn twists, say $X$, we denote by $\widehat X$ some lift of $\phi(X)$ to $\GL_n(\F)$. Such lifts are unique up to multiplication by a scalar matrix.

 In general, our aim is to show that $\phi(T)=1$ or that $\phi(T)$ commutes with $\phi(S)$ or that $\phi(P)$ commutes with $\phi(Q)$. In either of these cases, we deduce the fact that $\im \phi$ is abelian from \cref{killing dehn twists}.

 \medskip \textbf{Assume that $n=2$.} 
 Let us list the three possible Jordan normal forms of $\widehat T$ (we label the cases using Greek letters):
 \begin{align*}
   (\alpha) & & \left( \begin{array}{cc}
                \lambda & 0 \\
                0 & \lambda
               \end{array} \right),
& & (\beta) & & \left( \begin{array}{cc}
                \lambda & 1 \\
                0 & \lambda
               \end{array} \right),
& & (\gamma) & & \left( \begin{array}{cc}
                \lambda & 0 \\
                0 & \mu
               \end{array} \right)
 \end{align*}
 where $\mu \neq \lambda$. We fix a basis for $\F^2$ so that $\widehat T$ is precisely as depicted above.

 In case $(\alpha)$ we have $\phi(T) = 1$, and so we are done.

 \smallskip
 In case $(\beta)$ consider $T, P$ and $Q$. Since $T$ commutes with $P$ and $Q$, we know that $\widehat P$ and $\widehat Q$  permute the eigenspaces of $\widehat T$. But $\widehat T$ has a unique eigenspace, and so $\widehat P$ and $\widehat Q$ both preserve this eigenspace, and so are upper-triangular. Since $\widehat P$, $\widehat Q$ and $\widehat T$ are all conjugate, we see that every diagonal entry of $\widehat P$ and $\widehat Q$ is equal to $\lambda$ (we might have to change the lifts $\widehat P$ and $\widehat Q$ for this).  Thus  $\widehat P$ and $\widehat Q$ commute (this follows from the obvious direct calculation).

 \smallskip
 In case $(\gamma)$ we obtain a homomorphism $\rho \colon C_{\mcg{2}{b}}(T) \to 2$ where the codomain is the group of permutations of the eigenspaces of $\widehat T$ (note that $2$ stands for the cyclic group of order two; this is standard notation in finite groups theory). Using the braid relation $PQP=QPQ$ we see that the images $\rho(P)$ and $\rho(Q)$ coincide.

 If $\rho(P)$ is trivial, then $\widehat P$ and $\widehat Q$ are diagonal (as they preserve the eigenspaces of $\widehat T$), and therefore they commute.

 Suppose that $\rho(P)$ is not trivial. Note that this non-trivial permutation is induced by multiplying the eigenvalues of $\widehat T$ by $\omega \in \F$, where $[\widehat T, \widehat P] = \omega \I$, and $\omega^2=1$. By changing the lift $\widehat T$ we may assume that
 \[
  \widehat T = \left( \begin{array}{cc} \omega & 0 \\ 0 & 1 \end{array} \right)
 \]
 Now, as $\rho(P) \neq 1$, both $\widehat P$ and $\widehat Q$ must be of the form
 \[
  \left( \begin{array}{cc} 0 & \ast  \\  \ast & 0 \end{array} \right)
 \]
Changing the lifts $\widehat P$ and $\widehat Q$ so that, in particular, both matrices have determinant $\omega$ (which can be achieved by taking conjugates of $\widehat T$), we may write
 \[
  \widehat P =  \left( \begin{array}{cc} 0 & -\omega  \\  1 & 0 \end{array} \right) = \widehat Q
 \]
Now $\widehat P$ and $\widehat Q$ commute. This finishes the case $n=2$.

\medskip \noindent \textbf{Assume that $n=3$.}
This case is more involved, since there are six possibilities for the Jordan normal form of $\widehat T$. These are as follows:
\begin{align*}
   (i) & & \left( \begin{array}{ccc}
                \lambda & 0 & 0\\
                0 & \mu & 0 \\
                0 & 0 & \nu
               \end{array} \right),
& & (ii) & & \left( \begin{array}{ccc}
                \lambda & 0 & 0\\
                0 & \lambda & 0 \\
                0 & 0 & \lambda
               \end{array} \right),
& & (iii) & & \left( \begin{array}{ccc}
                \lambda & 0 & 0\\
                0 & \mu & 1 \\
                0 & 0 & \mu
               \end{array} \right),\\
   (iv) & & \left( \begin{array}{ccc}
                \lambda & 1 & 0\\
                0 & \lambda & 1 \\
                0 & 0 & \lambda
               \end{array} \right),
& & (v) & & \left( \begin{array}{ccc}
                \lambda & 0 & 0\\
                0 & \lambda & 1 \\
                0 & 0 & \lambda
               \end{array} \right),
& & (vi) & & \left( \begin{array}{ccc}
                \lambda & 0 & 0\\
                0 & \lambda & 0 \\
                0 & 0 & \mu
               \end{array} \right)\phantom{,}
 \end{align*}
 where $\lambda, \mu$ and $\nu$ are pairwise distinct. Again, we pick a basis for $\F^3$ so that $\widehat T$ is as depicted.

 Case $(i)$ is very similar to the case $(\gamma)$ above: we have a homomorphism $\rho$ from $C_{\mcg{2}{g}}(T)$ to the group of permutations of the eigenspaces of $\widehat T$. Since these permutations are given by multiplying the eigenvalues of $\widehat T$ by some $\omega \in \F$, we see that the image of our homomorphisms consists of $3$-cycles only and $\omega^3=1$. Therefore, using the braid relation between $P$ and $Q$ again, we see that $P$ and $Q$ have the same image $\rho(P)$ under this homomorphism.

 If $\rho(P)=1$ then $\widehat P$ and $\widehat Q$ are diagonal, and hence commute.

 If $\rho(P) \neq 1$ then we may take
 \[
  \widehat T = \left( \begin{array}{ccc}
                \omega^2 & 0 & 0\\
                0 & \omega & 0 \\
                0 & 0 & 1
               \end{array} \right)
 \]
and, without loss of generality,
\[
 \widehat P = \left( \begin{array}{ccc}
                0 & \delta & 0\\
                0 & 0 & \delta^{-1} \\
                1 & 0 & 0
               \end{array} \right), \
\widehat Q = \left( \begin{array}{ccc}
                0 & \epsilon & 0\\
                0 & 0 & \epsilon^{-1} \\
                1 & 0 & 0
               \end{array} \right)
\]
where $\delta, \epsilon \in \F \s- \{0\}$.
A direct computation shows that
\[
 \widehat P \widehat Q \widehat P = \left( \begin{array}{ccc}
                \delta \epsilon^{-1} & 0 & 0\\
                0 & 1 & 0 \\
                0 & 0 & \epsilon \delta^{-1}
               \end{array} \right), \ \widehat Q \widehat P \widehat Q = \left( \begin{array}{ccc}
                \delta^{-1} \epsilon & 0 & 0\\
                0 & 1 & 0 \\
                0 & 0 & \epsilon^{-1} \delta
               \end{array} \right)
\]
Thus, the braid relation between $\widehat P$ and $\widehat Q$ holds (a priori, it only had to hold up to multiplication by a scalar matrix), and informs us that $\epsilon^{-1} \delta = \delta^{-1} \epsilon$.
Therefore we have
\[
 \widehat P \widehat Q^{-1} = \left( \begin{array}{ccc}
                \delta \epsilon^{-1} & 0 & 0\\
                0 & \delta^{-1} \epsilon & 0 \\
                0 & 0 & 1
               \end{array} \right) = \left( \begin{array}{ccc}
                \delta \epsilon^{-1} & 0 & 0\\
                0 & \delta \epsilon^{-1} & 0 \\
                0 & 0 & 1
               \end{array} \right)
\]
If $\delta \epsilon^{-1}  =1$ then $\widehat P = \widehat Q$, and so we may apply \cref{killing dehn twists}. Otherwise, the eigenspaces of $\widehat P \widehat Q^{-1}$ cannot be permuted (as they have different dimensions). Therefore $\widehat S$ (up to multiplication by a scalar) has to be of the form
\[
 \left( \begin{array}{ccc}
                \ast & \ast & 0\\
                \ast & \ast & 0 \\
                0 & 0 & 1
               \end{array} \right)
\]
since $S$ and $PQ^{-1}$ commute in $\mcg{2}{b}$. Now, $\widehat P$ has to permute the eigenspaces of $\widehat S$, and so $\widehat S$ must be diagonal. But then $\widehat S$ commutes with $\widehat T$.

\smallskip
In the cases $(ii)$--$(vi)$, the eigenspaces of $\widehat T$, as well as the  more general kernels of matrices of the form $(\widehat T - \kappa \I)^n$, cannot be permuted. Thus, every lift $\widehat X$ of every element $X$  commuting in $\mcg{2}{b}$ with $T$ satisfies $[\widehat T, \widehat X]=1$ in $\GL_3(\F)$.

Let us now suppose that we are in the case $(ii)$. Here $\phi(T)=1$.

\smallskip
Let us consider the case $(iii)$. Since $P$ commutes with $T$, and by the remark above, $\widehat P$ commutes with $\widehat T -\lambda \I, \widehat T - \mu \I$, and $(\widehat T -\mu \I)^2$. Thus, $\widehat P$ (and also $\widehat Q$) is of the form
\[
 \left( \begin{array}{ccc}
                \ast & 0 & 0\\
                0 & \ast & \ast \\
                0 & 0 & \ast
               \end{array} \right)
\]
Since we know the Jordan normal form for $\widehat P$ (up to scalar multiplication), we see that $\widehat P$ (up to scalar multiplication again) is of the form
\[
 \left( \begin{array}{ccc}
                \lambda & 0 & 0\\
                0 & \mu & \ast \\
                0 & 0 & \mu
               \end{array} \right)
\]
and the same is true for $\widehat Q$. It now immediately follows from the obvious calculation that $\widehat P$ and $\widehat Q$ commute.

\smallskip
In case $(iv)$ we argue similarly, and see that $\widehat P$ and $\widehat Q$ must be upper-triangular. Since these matrices can only have one eigenvalue, we choose the lifts $\widehat P$ and $\widehat Q$ so that the matrices are unipotent. Thus, the braid relation between $\widehat P$ and $\widehat Q$ holds, and immediately shows that $\widehat P = \widehat Q$ (via a direct calculation), which finishes this case.

\smallskip Case $(v)$ is a little more involved, and we have to also use $\widehat S$. We choose $\widehat T$ and $\widehat S$ in such a way that their only eigenvalue is $1$. If the $1$-eigenspaces of $\widehat S$ and $\widehat T$ do not coincide, then $\widehat P$ and $\widehat Q$ have to be upper triangular (with respect to a suitable basis), since they have to preserve both the $1$-eigenspace of $\widehat T$ and of $\widehat S$. We then argue as in the previous case.

If the $1$-eigenspaces of $\widehat T$ and $\widehat S$ coincide, then $\widehat S$ is of the form
\[
 \left( \begin{array}{ccc}
                1 & 0 & 0\\
                \ast & 1 & \ast \\
                0 & 0 & 1
               \end{array} \right)
\]
A direct computation verifies that $\widehat T$ and $\widehat S$ commute.

\smallskip We are left with the case $(vi)$. We start precisely as in case $(v)$: choose lifts $\widehat T$ and $\widehat S$ in such a way that their eigenvalue with $2$-dimensional eigenspace is $1$. If these eigenspaces do not coincide, then $\widehat P$ and $\widehat Q$ are (after a change of basis of the $1$-eigenspace of $\widehat T$) of the form
\[
\left( \begin{array}{ccc}
                \ast & \ast & 0 \\
                0 & \ast & 0 \\
                0 & 0 & \ast
               \end{array} \right)
\]
Since $\widehat R$ preserves the eigenspaces of $\widehat T$ and $\widehat P$, the matrix $\widehat R$ is also of the form depicted above.

Up to possibly taking different lifts $\widehat P$, $\widehat Q$, and $\widehat R$, we know that the diagonal entries are $1,1$ and $\mu$, in some order. Using the braid relations between the pairs $P,Q$ and $Q,R$, we first see that the braid relations hold also for $\widehat P, \widehat Q$ and $\widehat Q, \widehat R$, and then we see that these diagonal entries appear in the same order in all three matrices $\widehat P$, $\widehat Q$, and $\widehat R$.

If the first two entries are both equal to $1$, then it is immediate (as before) that $\widehat P$ and $\widehat Q$ commute. Otherwise, the fact that $\widehat P$ and $\widehat Q$ satisfy the braid relation implies that $\widehat P = \widehat Q$ or that $\mu$ satisfies the equation $\mu^2-\mu-1=0$ (this follows from directly computing the two sides of the braid relation $\widehat P \widehat Q \widehat P = \widehat Q \widehat P \widehat Q$). In the latter case, the equation  $\mu^2-\mu-1=0$ guarantees that the braid relation holds also for the pair $\widehat P, \widehat R$. But $\widehat P$ and $\widehat R$ commute, and so $\phi(P) = \phi(R)$. Thus $\phi(S)$ commutes with $\phi(R)$, and we are done by \cref{killing dehn twists}.

We are left with the situation in which the $2$-dimensional eigenspaces of $\widehat T$ and $\widehat S$ coincide. Then we have $\widehat S$ of the form
\[
\left( \begin{array}{ccc}
                1 & 0 & \ast \\
                0 & 1 & \ast \\
                0 & 0 & \mu
               \end{array} \right)
\]
Now, arguing as above, we see that $\widehat S = \widehat T$ (in which case we are done), or $\mu$ satisfies the equation $\mu^2-\mu-1=0$. Since the pair $(T,S)$ can be conjugated in $\mcg{g}{b}$ to the pair $(R,S)$, the $2$-dimensional eigenspaces of $\widehat S$ and $\widehat R$ coincide. Thus, $\widehat R$ is of same form as $\widehat S$ (depicted above), and the equation $\mu^2-\mu-1=0$ tells us that the braid relation between $\widehat R$ and $\widehat T$ holds. But $\widehat R$ and $\widehat T$ commute, which forces $\widehat R = \widehat T$. This finishes the proof.
\end{proof}

\begin{prop}
\label{rep trick}
Let $V$ be a
vector space of finite dimension $n\geqslant 3$ over an algebraically closed field.
 Let $X \in \GL(V)$ be any matrix. Suppose that $X$ does not admit an eigenspace of dimension at least $n-1$. Then there exist subspaces $U$ and $U'$ of $V$ such that all of the following hold:
 \begin{enumerate}
  \item $U \leqslant U'$;
  \item the dimensions $\dim U, \dim U'/U$, and $\dim V/U'$ are at most $n-2$;
  \item both $U$ and $U'$ are preserved by $C_{\GL(V)}(X)$.
 \end{enumerate}
\end{prop}
\begin{proof}
 Let $\chi_X$ denote the characteristic polynomial of $X$.
 For a root $\lambda$ of $\chi_X$ we denote the corresponding eigenspace by $E_X^\lambda$.

Suppose that $\chi_X$ has only one root $\lambda$.
If $\dim E_X^\lambda = 1$ then the kernel of $(X - \lambda \I)^2$ has dimension $2$, and we are done by taking $U = E_X^\lambda$ and $U' = \ker (X - \lambda \I)^2$. Otherwise, we take $U = U' =E_X^\lambda$ -- recall that that $\dim U < n-1$ by assumption.

Suppose now that $\chi_X$ has exactly two distinct roots, $\lambda$ and $\mu$. For concreteness, let us assume that $\dim E_X^\mu \leqslant \dim E_X^\lambda$.
If $\dim E_X^\mu =1$ then $E_X^\mu \oplus E_X^\lambda \neq V$, since $\dim E_X^\lambda < n-1$ by assumption. Therefore we may take $U = E_X^\mu, U' = E_X^\mu \oplus E_X^\lambda$.

If $\dim E_X^\mu > 1$ then take
$U = U' = E_X^\mu$; we have $\dim E_X^\mu \leqslant \lfloor \frac n 2 \rfloor \leqslant n-2$, and so the dimensions are as required.

We are left with the possibility that $\chi_X$ has at least $3$ distinct roots, say $\lambda_1, \lambda_2$ and $\lambda_3$. 
 We may take $U = E_X^{\lambda_1}$ and $U' = E_X^{\lambda_1} \oplus E_X^{\lambda_2}$.
\end{proof}

Again, we follow the ideas of Korkmaz and deduce the following.

\begin{thm}
\label{main thm reps}
 Every projective representation of $\mcg{g}{b}$ with $g\geqslant 3$ and $b \geqslant 0$ in dimension less than $2g$ is trivial.
\end{thm}
\begin{proof}
Since we can extend scalars, we may assume that we are working over an algebraically closed field $\F$. Let $V$ denote the vector space  $\F^n$ with $n<2g$, and let $\phi \colon \mcg{g}{b} \to \PGL(V)$ be a projective representation.

Our proof is an induction on $g$, and a secondary induction on $n$. Note that when $n=1$ then $\PGL(V)$ is the trivial group, and the result follows.



To ease notation, set $T = T_1$ and $S = T_1'$ (see \cref{surface}). Let $\Gamma$ denote a subsurface of $\Sigma$ homeomorphic to $\Sigma_{g-1,b+1}$ which intersects the supporting curves of $T$ and $S$ trivially, but contains the support of the Dehn twists $T_i$ and $T_i'$ for $i>1$.
For every $X \in \mcg g b$ we choose a lift $\widehat X$ of $\phi(X)$ in $\GL(V)$.

Observe that for every $X \in \mcgop(\Gamma)$ we have $[\widehat T, \widehat X] = \lambda_X \I$ with $\lambda_X \in \F^\times$, and $\lambda_X$ being clearly independent of the choice of $\widehat X$ and $\widehat T$.
Define $h_T \colon \mcgop(\Gamma) \to \F^\times$ by $X \mapsto \lambda_X$. The commutator relation
\[
 [Z, XY] = [Z,X] X [Z,Y] X^{-1}
\]
immediately shows that $h_T$ is a homomorphism.


Note that when $g \geqslant 4$ then $h_T$ is trivial, since in this case $h_T$ is a homomorphism from a perfect group $\mcg {g-1}{b+1}$ (see \cref{perfect}) to an abelian group.
We are however also interested in the case $g=3$, and so we will need to worry about $h_T$ being non-trivial in the course of the proof.

We start by investigating the eigenspaces of $\widehat T \widehat {S}^{-1}$ in $V$.
Suppose first that $\widehat T \widehat {S}^{-1}$ does not admit an eigenspace of codimension at least $1$. This immediately implies that $n \geqslant 3$, and so we may apply  \cref{rep trick} to $X = \widehat T \widehat {S}^{-1}$. In this case, \cref{rep trick} gives us non-trivial subspaces $U$ and $U'$ of $V$ of codimension at least $2$.
The subspace $U$ is preserved by $\widehat X$ for every $X \in C_{\mcg g b}(TS^{-1})$, since $h_T = h_S$, as $S$ and $T$ are conjugate, and so $[\widehat T \widehat S^{-1}, \widehat X] = h_T(X) h_S(X)^{-1} \I = \I$.

When $g=3$, we apply \cref{reps base case} to the induced projective representations $U, U'/U$ and $V/U'$ of $\mcgop(\Gamma) \cong \mcg 2 {b+1}$. We conclude that these representations have abelian images, and so are trivial when restricted to the derived subgroup of $\mcgop(\Gamma)$. Hence $\phi$ restricted to this derived subgroup has nilpotent image. But the derived subgroup is perfect by \cref{perfect}, and so $\phi$ is trivial on the derived subgroup of $\mcgop(\Gamma)$. This implies that $\phi(PQ^{-1})=1$ (as $P$ and $Q$ are conjugate), and so $\phi$ is trivial by \cref{killing dehn twists,perfect}.

Now suppose that $g > 3$.
The inductive hypothesis tells us that the restricted projective representations
$U, U'/U$, and $V/U'$ of $\mcgop(\Gamma)$ are trivial.
Hence, $\phi$ restricted to $\mcgop(\Gamma)$ has nilpotent image. But $\mcgop(\Gamma) \cong \mcg{g-1}{b+1}$ is perfect by \cref{perfect}, and so $\phi$ is trivial on $\mcgop(\Gamma)$, and hence $\phi(T_2) = 1$. But $T_2 \in \mathcal S$, and the Dehn twists in $\mathcal S$ are pairwise conjugate and generate $\mcg  g b$, and so $\phi=1$.

\smallskip
Now suppose that $\widehat T \widehat S^{-1}$ admits an eigenspace $W$ of codimension at most $1$.
If the codimension is $0$, then $\phi(T) = \phi(S)$ and we are done by \cref{killing dehn twists}. Hence let us suppose that the codimension is equal to $1$.
Let $W'$ denote the eigenspace of codimension $1$ for $\widehat S^{-1} \widehat T$, which exists since $\widehat S^{-1} \widehat T$ and $ \widehat T \widehat S^{-1}$ are conjugate.

If $W \cap W' \neq W$, then we put $U = W \cap W'$ and $U' = W$ in the previous argument.

If $W = W'$ then we see that $W$ is also preserved by $\widehat S$, as $\widehat S^{-1} \widehat T$ and $\widehat T \widehat S^{-1}$ are related by the conjugation by $\widehat S$. Hence $\widehat T = \widehat S \cdot \widehat S^{-1} \widehat T$ preserves $W$.

Since $W$ is of codimension $1$, we can choose a basis for $V$ with exactly one vector, say $v$, lying outside of $
W$. Since $[\widehat T, \widehat P](v) = v + w$ with $w \in W$, we conclude that $h_T$ is trivial. Therefore the eigenspaces of $\widehat T$ are preserved by $\phi\big(\mcgop(\Gamma)\big)$.

Suppose that $\widehat T\vert_W$ is not central in $\GL(W)$. Then $\widehat T\vert_W$ has some non-trivial proper eigenspace $U$ in $W$. Now we can take $U=U$ and $U' = W$ in the previous argument.

We are left with the case of $\widehat T$ being central on $W$. If $\widehat T$ is central on $V$ then we are done by \cref{killing dehn twists}. Otherwise, $W$ is precisely an eigenspace of $\widehat T$. Using an analogous argument we show that it is also an eigenspace of $\widehat S$. As an eigenspace of $\widehat T$, the subspace $W$ is preserved by the image under $\phi$ of all elements in $\mcg{g}{b}$ which commute with $T$. But $\mcg{g}{b}$ is generated by such elements and $S$, and hence $W$ is a projective representation of $\mcg{g}{b}$ of smaller dimension than $V$. By induction, $W$ is the trivial projective representation. So is $V/W$, as it is one dimensional. We now finish the argument as before, using that every nilpotent quotient of $\mcg g b$ is trivial as $\mcg g b$ is perfect.
\end{proof}

\section{Minimal finite quotients of \texorpdfstring{$\mcg{g}{b}$}{MCG(Sigma g,b)}}

\subsection{Finite groups of Lie type}

Here we give an extremely rudimentary introduction to the finite groups of Lie type, restricted almost exclusively to the few facts and statements that we will require. For further details the reader is encouraged to consult the book of Gorenstein--Lyons--Solomon~\cite{Gorensteinetal1998}.

Note that we will denote the cyclic group of prime order $p$ by $p$. This is standard notation in finite group theory.

Let $r$ be a prime, and $q$ a power thereof.
The finite groups of Lie type over the field of $q$ elements are divided into finite collections (\emph{types}), each corresponding to a Dynkin diagram or a twisted Dynkin diagram;
the types $\A_n, \Atwo_n, \B_n,\typeC_n, \D_n$ and $\Dtwo_n$ are called \emph{classical}, and the types $\Btwo_2$, $\Dthree_4$, $\E_6$, $\Etwo_6$, $\E_7$, $\E_8$, $\typeF_4$, $\typeFtwo_4$, $\G_2$ and
$\Gtwo_2$ are called \emph{exceptional}.

As mentioned above, to each type we associate a finite family of finite groups; the members of such a family are called \emph{versions}. There are two special versions: the \emph{universal} one, which has the property that it maps homomorphically onto every other version with a central kernel, and the \emph{adjoint} version, which is the homomorphic image of every other version where the corresponding kernel is again central. The adjoint versions are simple with the following exceptions \cite[Chapter 3.5]{atlas}:
\begin{alignat*}{4}
\A_1(2) &\cong \Sym_3, & \A_1(3) &\cong \Alt_4, & \typeC_2(2) &\cong \Sym_6, &
\Atwo_2(2) &\cong 3^2 \rtimes Q_8, \\
\G_2(2) &\cong \Atwo_2(3) \rtimes 2,\phantom{xx} & \Btwo_2(2) &\cong 5 \rtimes 4,\phantom{xx} &
\Gtwo_2(3) &\cong \A_1(8) \rtimes 3,\phantom{xx} & \typeFtwo_4(2)
\end{alignat*}
where $Q_8$ denotes the quaternion group of order $8$, and $4$ denotes the cyclic group of that order. The group $\typeFtwo_4(2)$ contains an index $2$ subgroup $\mathrm{T} = \typeFtwo_4(2)'$, known as the Tits group, which is simple. For the purpose of this paper, we treat $\mathrm{T}$ as a finite group of Lie type.

Each finite group of Lie type has a \emph{rank}, which is simply the number of vertices of the corresponding Dynkin diagram (or, equivalently, the index appearing as a subscript of the type).

Groups of types $\Btwo_2$ and $\typeFtwo_4$ are defined only over fields of order $2^{2m+1}$ while groups of type $\Gtwo_2$ are defined only over fields of order $3^{2m+1}$. All groups of all other types are defined over all finite fields.

For reference, we also recall the following additional exceptional isomorphisms
\[\A_1(4) \cong \A_1(5) \cong \Alt_5, \; \; \A_1(9) \cong \Alt_6, \; \; \A_1(7) \cong \A_2(2), \; \; \A_3(2) \cong \Alt_8, \; \; \Atwo_3(2) \cong \typeC_2(3)\]
In addition, $\B_n(2^m) \cong \typeC_n(2^m)$ for all $n \geqslant 3$ and $m \geqslant 1$.


The adjoint version of a classical group over $q$ comes with a natural projective module over an algebraically closed field in characteristic $r$; the dimensions of these modules are taken from \cite[Table 5.4.C]{KleidmanLiebeck1990} and listed in \cref{dynkinclassical}. Note that these projective modules are irreducible. \cref{dynkinclassical} lists also the necessary conditions on the ranks for a given type.

The Dynkin diagrams $\B_2$ and $\typeC_2$ coincide, and so we do not talk about groups of type $\typeC_2(q)$. However, when it comes to finding the smallest projective modules, one should consider the adjoint version of $\B_2(q)$ as the projective symplectic group $\mathrm{S}_{4}(q)$, rather then the projective orthogonal group $\mathrm{O}_{5}(q)$. This technical point will be however of no bearing for us.

\begin{table}
\centering\begin{tabular}{r | c | c |  c}
                        &                                                       &                       &Classical\\
Type            & Conditions                                    & Dimension     &isomorphism \\ \hline \hline
$\A_n(q)$               &  $n\geqslant 1$                               & $n+1$                 &$\mathrm{L}_{n+1}(q)$\\
$\Atwo_n(q)$    &  $n\geqslant 2$        & $n+1$        &U$_{n+1}(q)$\\
$\B_n(q)$               & $n\geqslant 2$                                & $2n+1$                &$\mathrm{O}_{2n+1}(q)$\\
$\typeC_n(q)$   & $n\geqslant 3$                                & $2n$          &$\mathrm{S}_{2n}(q)$\\
$\D_n(q)$       & $n\geqslant 4$                                & $2n$          &$\mathrm{O}_{2n}^+(q)$\\
$\Dtwo_n(q)$    & $n\geqslant 4$                                & $2n$          &$\mathrm{O}_{2n}^-(q)$\\ \hline
\end{tabular}
\caption{The classical groups of Lie type}
\label{dynkinclassical}
\end{table}

\smallskip
A \emph{parabolic} subgroup of $K$ is any subgroup containing a Borel subgroup of $K$, that is the normaliser of a Sylow $r$-subgroup of $K$.

\begin{prop}[Jordan decomposition]
Let $K$ be a finite group of Lie type defined over a field of characteristic $r$. For every element $x \in K$ we have unique elements $x_s, x_u \in K$, such that
\[x = x_s x_u  = x_u x_s\]
and the order of $x_u$ is a power of $r$ ($x_u$ is \emph{unipotent}), and the order of $x_s$ is coprime to $r$ ($x_s$ is \emph{semisimple}).
\end{prop}

In fact, the above works for any element of a finite group, and any prime $r$. We stated it in such a form to emphasise the relation to algebraic groups.

We now look at the first of the two structural results about finite groups of Lie type that we will need.

\begin{thm}[Borel--Tits~{\cite[Theorem 3.1.3(a)]{Gorensteinetal1998}}]
\label{borel--tits}
 Let $K$ be a finite group of Lie type in characteristic $r$, and let $R$ be a non-trivial $r$-subgroup of $K$. Then there exists a proper parabolic subgroup $P < K$ such that $R$ lies in the normal $r$-core of $P$, and $N_K(R) \leqslant P$.
\end{thm}

Recall that the normal $r$-core is the maximal normal $r$-subgroup.


\begin{thm}[Levi decomposition~{\cite[Theorem 2.6.5(e,f,g), Proposition 2.6.2(a,b)]{Gorensteinetal1998}}]
\label{levi factors}
 Let $P$ be a proper parabolic in a finite group $K$ of Lie type in characteristic $r$.
 \begin{enumerate}
  \item Let $U$ denote the normal $r$-core of $P$ (note that $U$ is nilpotent).  There exists a subgroup $L \leqslant P$, such that $L\cap U = \{1\}$ and $L U = P$.
  \item $L$ (the \emph{Levi factor}) contains a normal subgroup $M$ such that $L/M$ is abelian of order coprime to $r$.
  \item $M$ is isomorphic to a central  product of finite groups of Lie type (the \emph{simple factors of $L$}) in characteristic $r$ such that the sum of the ranks of these groups is lower than the rank of $K$.

 \end{enumerate}
 \end{thm}

The following is the second structural result that we will need.

\begin{thm}[{\cite[Theorem 4.2.2]{Gorensteinetal1998}}]
\label{semisimple autos}
Let $K$ be an adjoint version of a finite group of Lie type defined in characteristic $r$.
Let $x \in K$ be an element of prime order $p$ with $p\neq r$, and let $C$ denote its centraliser in $K$. Then
\begin{enumerate}
 \item The group $C$ contains a normal subgroup $C^0$ (the \emph{connected centraliser}), such that $C/C^0$ is an elementary abelian $p$-group.
 \item The group $C^0$ contains an abelian subgroup $T$ (the \emph{torus}) and a normal subgroup $L$ such that $C^0 = LT$.
 \item The group $L$ is a central product of subgroups $L_1, \dots, L_s$ (the \emph{Lie factors}); each Lie factor $L_i$ is a finite group of Lie type in characteristic $r$.
 \item The sum of the ranks of the Lie factors is bounded above by the rank of $K$.

\end{enumerate}
\end{thm}

\begin{rmk}
From the aforementioned results in \cite{Gorensteinetal1998}, and from \cite{DeriziotisMichler1987} and \cite{Shinoda75} we can additionally deduce that
\begin{enumerate}
\item the groups of type $\G_2(q)$ and $\Gtwo_2(q)$ are never simple factors of a Levi factor of a proper parabolic subgroup of any group of Lie type;
\item when $K$ is of type $\Dthree_4(q)$, then the simple factors of the Levi factor and the Lie factors are isomorphic to $\A_1$,  $\A_2$, or $\Atwo_2$; and
\item when $K$ is of type $\typeFtwo_4(q)$, then the simple factors of the Levi factor and  the Lie factors are isomorphic to $\A_1$, $\Btwo_2$, and $\Atwo_2$.
\end{enumerate}
\end{rmk}

\subsection{Alternating and classical groups}
In this subsection we prove that $\sp_{2g}(2)$ is the smallest quotient of $\mcg g b$ among the alternating and classical groups; \cref{main thm} will then follow for sufficiently large $g$. The proof is supported on two pillars: \cref{main thm reps} and the following result of Berrick--Gebhardt--Paris:

\begin{thm}[{\cite[Theorem 4]{Berricketal2014}}]
 \label{bgparis}
 Let $g \geqslant 3$ and $b\geqslant 0$. Up to conjugation, the group $\mcg{g}{b}$ contains a unique subgroup of index $2^{g-1}(2^g-1)$. Also, it does not contain any proper subgroups of smaller index.
\end{thm}
Note that we have rephrased the theorem in a way suitable to our needs -- in fact Berrick--Gebhardt--Paris prove more, since they also establish the index of the second smallest subgroup, and give lower bounds for the index of the third one.

\begin{cor}
\label{uniqueness}
Let $g \geqslant 3$ and $b \geqslant 0$.
Any epimorphism $\mcg{g}{b} \to \sp_{2g}(2)$ is obtained by postcomposing the natural map $\mcg{g}{b} \to \sp_{2g}(2)$ with an automorphism of $\sp_{2g}(2)$.
\end{cor}
Let us remark here that for $g \geqslant 3$, the group $\sp_{2g}(2)$ has trivial outer automorphism group \cite[Section 3.5.5]{raw}.
\begin{proof}
Let us start with the natural epimorphism $\nu \colon \mcg{g}{b} \to \sp_{2g}(2)$ obtained by first gluing discs to the boundary components (`capping'), then abelianising $\pi_1(\Sigma_{g})$, and then reducing the entries of matrices modulo $2$. Since $\sp_{2g}(2)$ has a (maximal) subgroup $M$ of index $2^{g-1}(2^g-1)$ (see \cite[Table 5.2.A]{KleidmanLiebeck1990}), we conclude that $\nu^{-1}(M)$ is the unique (up to conjugation) subgroup of $\mcg{g}{b}$ of index $2^{g-1}(2^g-1)$. Crucially, $\nu^{-1} (M)$ contains
the kernel $\ker \nu$.

Now let $\phi \colon \mcg{g}{b} \to \sp_{2g}(2)$ be any epimorphism. By uniqueness, we see that $\phi^{-1}(M)$ is conjugate to $\nu^{-1}(M)$, and so, in particular, contains the normal subgroup $\ker \nu$ of $\mcg{g}{b}$. Therefore $\phi(\ker \nu)$ is a proper normal subgroup of $\sp_{2g}(2)$. But the latter group is simple, and therefore $\phi(\ker \nu) = \{1\}$. We conclude that $\phi$ factors through $\nu$. Since the images of $\nu$ and $\phi$ have equal cardinality, we immediately conclude that $\phi$ is equal to $\nu$ followed by an automorphism of $\sp_{2g}(2)$.
\end{proof}

\begin{prop}
\label{large g}
 For $g \geqslant 3$, the smallest non-trivial quotient of $\mcg{g}{b}$ among alternating groups and classical groups of Lie type is $\sp_{2g}(2)$.
\end{prop}
\begin{proof}
Let $K$ denote the smallest non-trivial quotient of $\mcg{g}{b}$ among the alternating groups and the classical groups of Lie type.  
\cref{bgparis} tells us that $\mcg{g}{b}$ cannot act on fewer than $2^{g-1}(2^g-1)$ points, and thus if $K$ is an alternating group then its rank is bounded below by $2^{g-1}(2^g-1)$. But using Stirling's approximation we see that
\begin{align*}
 \vert \Alt_{2^{g-1}(2^g-1)} \vert &= \frac 1 2 \big(2^{g-1}(2^g-1)\big)! \\
 &\geqslant \frac 1 2 (2 \pi)^{\frac 1 2} \big(2^{g-1}(2^g-1)\big)^{2^{g-1}(2^g-1) + \frac 1 2} e^{-2^{g-1}(2^g-1)} \\
 &\geqslant \big( {2^{g-1}(2^g-1)} e^{-1} \big)^{2^{g-1}(2^g-1) }  \\
 &> \big( {2^{g-3}(2^g-1)}  \big)^{2^{g-1}(2^g-1) }  \\
 &> \big( {2^{g-3}(2^g-1)}  \big)^{9g }  \\
 &>  {2^{9g^2-27}\prod_{i=1}^g (2^g-1)^9}    \\
 &>  {2^{g^2}\prod_{i=1}^g (2^{2i}-1)}    \\
 &= \vert \sp_{2g}(2) \vert
\end{align*}
for $g \geqslant 3$.

Now let us assume that $K$ is an adjoint version of a classical group of Lie type.
Combining the assumption that $\vert K \vert \leqslant \vert \sp_{2g}(2) \vert$ with \cref{main thm reps}, we conclude that $K$ is isomorphic to $\D_g(2)$, $\Dtwo_g(2)$, or $\B_g(2) \cong \typeC_g(2) \cong \sp_{2g}(2)$.
But $\Dtwo_g(2)$ has a subgroup of index $2^{g-1}(2^g-1)-1$ (see \cite[Table 5.2.A]{KleidmanLiebeck1990}), and hence cannot be a quotient of $\mcg{g}{b}$ by \cref{bgparis}; similarly, $\D_g(2)$ has a subgroup of index $2^{g-1}(2^g-1)$ and thus it is ruled out by the uniqueness part of \cref{bgparis}, since $\D_g(2)$ is simple and not isomorphic to $\sp_{2g}(2)$, and we may argue exactly as  in \cref{uniqueness}.
\end{proof}

\begin{rmk}
At this point we can already say that
for sufficiently large $g$, the smallest non-trivial quotient of $\mcg{g}{b}$ is $\sp_{2g}(2)$, and the quotient map is obtained by postcomposing the natural map $\mcg{g}{b} \to \sp_{2g}(2)$ with an automorphism of $\sp_{2g}(2)$.
For large enough $g$ the sporadics are excluded by any one of \cref{main thm reps,bgparis}, and the exceptional groups of Lie type are excluded by \cref{main thm reps}, since the smallest dimension of a non-trivial projective representation of such a group is bounded above by $248$. We also use \cref{uniqueness}.
\end{rmk}

\subsection{Exceptional groups}

To deal with exceptional groups we will develop a technique that applies to all finite groups of Lie type, excluding $\G_2(q)$ and $\Gtwo_2(q)$. For classical groups it does not however surpass \cref{main thm reps} in applicability.

\begin{dfn}
 Let $K$ be a finite group. Given an integer $m\geqslant 2$, The \emph{$m$-rank} of $K$ is defined to be the largest integer $n$ such that $m^n$, the $n$-fold direct product of the cyclic group of order $m$, embeds into $K$.
\end{dfn}

\begin{lem}
\label{p-rank}
Let $K$ be the adjoint version of 
a group of Lie type over the field of size $q$.
Let $p$ be an odd prime coprime to $q$.
Then the $p$-rank of $K$ is bounded above by the rank of $K$.
\end{lem}
\begin{proof}
Using \cite[Theorem 4.10.3b]{Gorensteinetal1998} we see that the $p$-rank of $K$ is bounded above by the number $n_{m_0}$, defined to be the multiplicity of the cyclotomic polynomial associated to $m_0$ in the order of $K_0$ thought of as a polynomial in $q$. Here $K_0$ stands for the group of inner diagonal automorphisms of $K$, and $m_0$ is the multiplicative order of $q$ modulo $p$.

Note that $K_0$ has the same order as the universal  version of $K$, and these orders are given in \cite[Table 2.2]{Gorensteinetal1998}.
They are always of the form
\[
 q^N \prod_{i=1}^k (q^{n_i} -\omega_i)
\]
where $N$ is a natural number, $k$ is the rank of $K$, and $\omega_i$ is a root of unity (of order at most $3$). (Note that, following the convention of \cite{Gorensteinetal1998}, for Suzuki--Ree groups i.e. types $\Btwo_2, \Gtwo_2$, and $\typeFtwo_4$, we take the square root of $q$ in place of $q$ in the above formula.)
Given a cyclotomic polynomial $\Phi$, none of the polynomials of the form $q^{n_i} -\omega_i$ is divisible by $\Phi^2$. Thus the multiplicity $n_{m_0}$ is bounded above by $k$.
\end{proof}

\begin{prop}
\label{small p-rank}
Let $g \geqslant 2$ and $b \geqslant 0$.
Let $\phi \colon \mcg{g}{b} \to K$ be a homomorphism to a finite group $K$ with $\phi(T_1) \neq 1$.
\begin{enumerate}
 \item Let $m\geqslant 2$ divide the order of $\phi(T_1)$. If the $m$-rank of $K$ is less than $g$ then some $\phi(T_1)^n$ with $\phi(T_1)^n \neq 1$ is central in $\im \phi$.
 \item If $\phi(T_1)^2=1$ and the $3$-rank of $K$ is less than $g$, then $\phi(T_1)$ is central in $\im \phi$.
\end{enumerate}
\end{prop}
\begin{proof}
\noindent \textbf{(1)}
Let  the order of $\phi(T_1)$ be $km$.
Consider
\[
 Z = \langle \{ \phi(T_1)^k, \dots, \phi(T_g)^k \} \rangle
\]
This is the image in $K$ of a group isomorphic to $m^g$, and so, since the $m$-rank of $K$ is less than $g$, for some $j$ we have
\[
 \phi(T_j)^{n} = \prod_{i\neq j} \phi(T_i)^{n_i}
\]
with $n, n_i \in \Z$ and $\phi(T_j)^{n} \neq 1$.

Now consider $\phi(T_j')$. Since $T_j'$ commutes with every $T_i$ with $i\neq j$, we immediately see that $\phi(T_j')$ commutes with $\phi(T_j)^n$. But then $\phi(T_j)^n$ commutes with the image under $\phi$ of every generator of $\mcg{g}{b}$ from $\mathcal S$, and so $\im \phi$ lies in the centraliser of $\phi(T_j)^n$. Using conjugation, we obtain the same result for $\phi(T_1)^n$, as required.

\smallskip
\noindent \textbf{(2)}
Suppose that $\phi(T_1)$ has order $2$. Then using the braid relation
 \[T_1 T_1' T_1 = T_1' T_1 T_1'\]
 and noting that $\phi(T_1')$ also has order $2$ as $T_1$ and $T_1'$ are conjugate, we conclude that $\phi(S_1)$ is either trivial, or has order $3$, where $S_1 = T_1 T_1'$.

 If $\phi(S_1) = 1$ then $\phi(T_1)$ commutes with $\phi(T_1')$. Since $T_1$ commutes with all the other generators in $\mathcal S$, it is immediate that $\phi(T_1)$ is central.

 If $\phi(S_1)$ has order $3$, then consider
\[
 W = \langle \{ \phi(S_1), \dots, \phi(S_g) \} \rangle
\]
where $S_i = T_i T_i'$ has order $3$ for each $i$. Using the assumption on the $3$-rank of $K$, and the fact that $3$ is prime, without loss of generality we have
\[
 \phi(S_1)= \prod_{i>1} \phi(S_i)^{n_i}
\]
with $n_i \in \{0,1,2\}$. Similarly as before, we conclude that $\phi(T_1)$ commutes with $\phi(S_1)$. But $S_1 = T_1 T_1'$, and $T_1$ commutes with itself, and so $\phi(T_1)$ commutes with $\phi(T_1')$ and we are done as before.
\end{proof}

\begin{cor}
\label{ranks cor}
Let $g \geqslant 2$ and $b \geqslant 0$.
 Let $K$ be a finite non-trivial group with trivial centre and  with $m$-rank less than $g$ for every $m \geqslant 3$. Then $K$ is not a quotient of $\mcg{g}{b}$.
\end{cor}
\begin{proof}
Suppose that $\phi \colon \mcg{g}{b} \to K$ is a quotient map. Since $K$ is non-trivial and has trivial centre, we have $\phi(T_1) \neq 1$ by \cref{killing dehn twists}, and we also see that $\phi(T_1)$ cannot be central in $\im \phi$. Therefore, by \cref{small p-rank}(1), the order of $\phi(T_1)$ is $2$. But this contradicts \cref{small p-rank}(2).
\end{proof}

\begin{thm}
\label{lie type}
Let $b \geqslant 0$.
Let $g \geqslant 4$ and $K$
be a version of a group of Lie type of rank less than $g$, or let $g=3$ and $K$ be a version of a group of Lie type of rank less than $g$ with the exception of $\G_2(q)$ and $\Gtwo_2(q)$.
Then every homomorphism 
 $\mcg{g}{b} \to K$ is trivial.
\end{thm}
\begin{proof}
Let $\phi$ denote such a homomorphism. The proof is an induction on $g$.
Since $\mcg{g}{b}$ is perfect we may assume that $K$ is the adjoint version.


For the base case we use \cref{main thm reps}: it immediately deals with the cases $\A_1(q), \A_2(q), \Atwo_2(q)$, and $\B_2(q)$. It actually also rules out the case $\Btwo_2(q)$, since the adjoint version of this type has a faithful projective representation in dimension $4$.

\smallskip
Suppose that $g \geqslant 4$.
Let $a = \phi(T_1)$. Let us look at the Jordan decomposition $a = a_u a_s$, where $a_u$ is the unipotent part and $a_s$ the semi-simple part. Let us first assume that $a_u \neq 1$. Since the Jordan decomposition is unique, $\phi\big(C_{\mcg{g}{b}}(T_1)\big)$ commutes with $a_u$, and therefore, by \cref{borel--tits}, there exists a proper parabolic subgroup $P$ containing $\phi(C_{\mcg{g}{b}}(T_1))$. Using \cref{levi factors} (and its notation) we see that the parabolic subgroup $P$ contains a nilpotent normal subgroup $U$ such that $P/U = L$.
Note that we have an epimorphism $\mcg{g-1}{b+2} \to C_{\mcg{g}{b}}(T_1)$ obtained by identifying two boundary components of $\Sigma_{g-1,b+2}$ and declaring the newly obtained simple closed curve to be the curve underlying $T_1$ (it is easy to see that this is an epimorphism; alternatively, see~\cite[Section 2.3]{AramayonaSouto2012}).
Let
\[
 \psi \colon \mcg{g-1}{b+2} \to P/U = L
\]
denote the homomorphism obtained by composing the above map with the restriction of $\phi$.
 The Levi factor $L$ contains a normal subgroup $M$ such that $L/M$ is abelian.
The group $\mcg{g-1}{b+2}$ is perfect, and therefore $\im \psi \leqslant M$. The group $M$ is a central product of finite groups of Lie type  whose rank is strictly smaller than the rank of $K$. The inductive hypothesis tells us that $\psi$ is trivial.
Therefore $\phi\big(C_{\mcg{g}{b}}(T_1)\big)$ is trivial as $U$ is nilpotent and $\mcg{g-1}{b+2}$ is perfect. Thus $\phi$ is the trivial homomorphism, as $C_{\mcg{g}{b}}(T_1)$ contains $T_2 \in \mathcal S$, and all generators in $\mathcal S$ are conjugate.

Note that the exclusion of $\G_2(q)$ and $\Gtwo_2(q)$ for $g=3$ is not a problem for the induction, since these groups do not appear as proper parabolic subgroups.


\smallskip
Now suppose that $a_u = 1$.  We may assume that $a_s \neq 1$, as otherwise we have $\phi(T_1)=1$ which trivialises $\phi$. Pick $n \in \Z$ so that $a_s^n$ is of prime order $p$ in $K$.
Note that $p$ does not divide the characteristic of the ground field of $K$, and so by \cref{p-rank,small p-rank} we see that $\im \phi$ centralises $a_s^n$.

Let $C$ denote the centraliser of $a_s^n$ in $K$. We apply \cref{semisimple autos}, and use the notation thereof. The centraliser $C$ contains a normal subgroup $C^0$ such $C/C^0$ is abelian. But $\mcg{g}{b}$ is perfect (as $g \geqslant 3$), and so $\im \phi \leqslant C^0$. Now $C^0/L$ is abelian, and so again we see that $\im \phi \leqslant L$. Now we need to consider two cases: either each of the Lie factors $L_1,\dots, L_s$ has rank strictly smaller than $K$, or there is only one such factor of rank equal to $K$. In the former case we apply the inductive hypothesis to $\phi$ followed by a projection $L \to L_i$ for each $i$, and conclude that $\phi$ is trivial. In the latter case we are satisfied with the conclusion that $\im \phi \leqslant L$, since $L \leqslant C$, and so is of smaller cardinality than $K$, as $K$ has no centre.
We now run a secondary induction on the order of $K$.
Again, the exclusion of $\G_2(q)$ and $\Gtwo_2(q)$ is not a problem, since it occurs only for $g=3$.
\end{proof}

\begin{rmk}
\label{Dthree}
In fact the above proof also shows that if $K$ is of type $\Dthree_4(q)$ or $\typeFtwo_4(q)$ then every homomorphism $\mcg{g}{b} \to K$ is trivial for every $g \geqslant 3$ and $b \geqslant 0$. The reason for this is that the simple factors of the Levi factor of any proper parabolic subgroup of $K$ or the Lie factors of centralisers of semi-simple elements are of type $\A_1$, $\A_2$, $\Atwo_2$ or  $\Btwo_2$, and so all of rank at most $2$.
\end{rmk}

\begin{cor}
\label{exceptionals}
None of the groups of exceptional type is the smallest quotient of some $\mcg{g}{b}$ with $g \geqslant 3$, $b\geqslant 0$.
\end{cor}
\begin{proof}
Let $K$ be a finite exceptional group of Lie type. Since $\mcg{g}{b}$ is perfect, we need only consider the adjoint versions. \cref{lie type} tells us that $K$ cannot be a quotient of $\mcg{g}{b}$ unless $g$ is bounded above by the rank of $K$, or unless $g=3$ and $K$ is the adjoint version of $\G_2(q)$ or $\Gtwo_2(q)$.
The types $\Dthree_4$  and $\typeFtwo_4(q)$ are excluded by \cref{Dthree}.

A direct computation of orders (see \cite[Table 2.2]{Gorensteinetal1998}) tells us that
\begin{enumerate}
\item $\vert \Gtwo_2(27) \vert > \vert \G_2(3) \vert > \vert \sp_6(2) \vert$
\item $\vert \typeF_4(2) \vert  > \vert \sp_8(2) \vert$
\item $\vert \E_6(2) \vert > \vert \Etwo_6(2) \vert > \vert \sp_{12}(2) \vert$
\item $\vert \E_7(2) \vert > \vert \sp_{14}(2) \vert$
\item $\vert \E_8(2) \vert > \vert \sp_{16}(2) \vert$
\end{enumerate}
and this concludes the proof, since the smallest member of each of the families except $\G_2$ and $\Gtwo_2$  is the group defined over the field of $2$ elements. For $\G_2$, the smallest simple member of the family is $\G_2(3)$; for $\Gtwo_2$, the smallest simple member of the family is $\Gtwo_2(27)$.
\end{proof}

\subsection{Sporadic groups}

In this section we prove the main result.

\begin{thm}
 \label{main thm}
 Let $K$ be a non-trivial finite quotient of $\mcg{g}{b}$, the mapping class group of a connected orientable surface of genus $g\geqslant 3$ with $b$ boundary components. Then either $\vert K \vert > \vert \sp_{2g}(2) \vert$, or $K \cong \sp_{2g}(2)$ and the quotient map is obtained by postcomposing the natural map $\mcg{g}{b} \to \sp_{2g}(2)$ with an automorphism of $\sp_{2g}(2)$.
\end{thm}
\begin{proof}
Without loss of generality, we may take $K$ to be a smallest quotient of $\mcg{g}{b}$.
\cref{exceptionals} tells us that $K$ is not an exceptional group of Lie type, and by \cref{large g} we see that we need only rule out the sporadic groups.

We will go through the list of sporadic groups in order of increasing cardinality (see \cref{spo}). Considering a group $K$, we define $g(K)$ to be the smallest $g$ such that $\vert K \vert < \sp_{2g}(2)$.
Our aim is to show that $K$ is not the quotient of any $\mcg{g}{b}$ with $g \geqslant g(K)$.

Suppose that we have achieved our goal for all groups smaller than $K$. If we show that $K$ is not a quotient of $\mcg {g(K)} b$ for any $b$, then we can also conclude that it is not a quotient of any $\mcg{g}{b}$ with $g > g(K)$: suppose that $K$ is a quotient of such a $\mcg{g}{b}$. We have a homomorphism $\mcg{g(K)}1 \to \mcg{g}{b}$, and its image in $K$ is a subgroup of $K$, and hence either $K$ itself, which is impossible, or a proper subgroup of $K$. But this is also impossible, since we have shown that groups smaller than $K$ cannot be quotients of $\mcg{g(K)}1$.

Let us now start going through the list, and suppose that we consider a group $K$;
at this point we already know that all non-trivial groups smaller that $K$ are never quotients of $\mcg {g(K)} b$. Thus, when looking at $K$, we may use the assumption that all homomorphisms from $\mcg {g(K)} b$ to groups smaller than $K$ are trivial.

\begin{table}[h]\centering\begin{tabular}{l || r }
$K$                                                     &Order of $K$\\
\hline
$\sp_{4}(2)$                                     &720\\
\hline
$\mathrm{M}_{11}$                               &7920\\
$\mathrm{M}_{12}$                                       &95040\\
$\mathrm{J}_1$                                  &175560\\
$\mathrm{M}_{22}$                                       &443520\\
$\mathrm{J}_2$                                  &604800\\
$\sp_{6}(2)$                                     &1451520\\
\hline
$\mathrm{M}_{23}$                                       &10200960\\
$\mathrm{HS}$                                   &44352000\\
$\mathrm{J}_3$                                  &50232960\\
$\mathrm{M}_{24}$                                       &244823040\\
$\mathrm{M^cL}$                                 &898128000\\
$\mathrm{He}$                                   &4030387200\\
$\sp_{8}(2)$                                     &47377612800\\
\hline
$\mathrm{Ru}$                                   &145926144000\\
$\mathrm{Suz}$                                  &448345497600\\
$\mathrm{O'N}$                                  &460815505920\\
$\mathrm{Co}_3$                                 &495766656000\\
$\mathrm{Co}_2$                                 &42305421312000\\
$\Fi_{22}$                                      &64561751654400\\
$\mathrm{HN}$                                   &273030912000000\\
$\sp_{10}(2)$                                    &24815256521932800\\
\hline
$\mathrm{Ly}$                                           &51765179004000000\\
$\mathrm{Th}$                                           &90745943887872000\\
$\Fi_{23}$                                      &4089470473293004800\\
$\mathrm{Co}_1$                         &4157776806543360000\\
$\mathrm{J}_4$                          &86775571046077562880\\
$\sp_{12}(2)$                                    &208114637736580743168000\\
\hline
$\Fi_{24}'$                                     &1255205709190661721292800\\
$\sp_{14}(2)$                                    &27930968965434591767112450048000\\
\hline
$\mathrm{B}$                            &4154781481226426191177580544000000\\
$\sp_{16}(2)$                                    &59980383884075203672726385914533642240000\\
\hline
$\sp_{18}(2)$                                    &2060902435720151186326095525680721766346957783040000\\
\hline
$\mathrm{M}$            &808017424794512875886459904961710757005754368000000000\\
$\sp_{20}(2)$                                    &1132992015386677099994486205757869431795095310094129168384000000\\
\hline
\end{tabular} \caption{}
\label{spo}
\end{table}

We start by invoking \cref{ranks cor}, and identify those groups $K$ whose $m$-rank, where $m$ is equal to $4$ or an odd prime, is less than $g(K)$ (this is equivalent to having all $m$-ranks less than $g(K)$ for all $m \geqslant 3$). The prime ranks of sporadic groups are known and listed in \cite[Table 5.6.1]{Gorensteinetal1998}; it turns out that for $m\geqslant 5$ the ranks are always bounded above by $g(K)-1$.
If the $3$-rank or the $4$-rank of $K$ is at least $g(K)$, then
$K$ is
listed in \cref{spo2}.
(The $4$-rank was computed using GAP~\cite{gap} -- because of the 
complexity of the problem of determining the $4$-rank, we have only computed it for some sporadic groups. Thus the list may contain groups whose $3$- and $4$-ranks are in fact smaller than $g(K)$. It may also contain traces of nuts.) At this point we already know that the remaining $15$ sporadic groups are not the smallest quotients of $\mcg{g}{b}$.

Suppose that $K$ is one of the $11$ groups in \cref{spo2}, and let
 \[\phi \colon \mcg{g(K)}{b} \to K\]
 denote the putative quotient map. We may assume that $\phi(T_1)$ is not trivial and of order divisible only by $2$ and $3$ (the latter assumption is justified by \cref{small p-rank}). Thus some power of $\phi(T_1)$ will have order exactly $2$ or $3$, and the centraliser of $T_1$ in $\mcg{g(K)}{b}$ is an epimorphic image of $\mcg{g(K)-1}{b+2}$. Thus, if we know that every homomorphism from $\mcg{g(K)-1} b$ to any simple non-abelian factor of a centraliser of an element of order $2$ or $3$ in $K$ is trivial, then we may conclude that $\phi$ is trivial.

For the remaining $11$ sporadic groups, \cref{spo2} lists the non-abelian simple factors of centralisers of elements of order $2$ or $3$ (the information is taken from \cite[Table 5.3]{Gorensteinetal1998}).
(To simplify the notation, in the sequel we use type to denote the corresponding adjoint version.)

\begin{table}[h]
	\centering
		\begin{tabular}{c c c c c}\hline
   				& 		& \multicolumn{3}{c}{Simple factors of centralisers of elements of order $2$ or $3$}\\
$K$				&$g(K)$	&\!\!\!\!\!Alternating				& \!\!\!\!\!\!\!\!\!\! Lie type				&\!\!\!\!\!Sporadic\\
\hline
$\mathrm{M^cL}$	&4		&\!\!\!\!\!$\Alt_5$, $\Alt_8$ 			&\!\!\!\!\!\!\!\!\!\!--									&\!\!\!\!\!--\\
\hline
$\mathrm{Suz}$	&5		&\!\!\!\!\!$\Alt_6$					&\!\!\!\!\!\!\!\!\!\!$\A_2(4)$, $\Atwo_3(3)$, $\Dtwo_3(2)$		&\!\!\!\!\!--\\
$\mathrm{Co}_3$	&		&\!\!\!\!\!$\Alt_5$, $\Alt_6$			&\!\!\!\!\!\!\!\!\!\!$\A_1(8)$, $\typeC_3(2)$				&\!\!\!\!\!$\mathrm{M}_{12}$\\
$\mathrm{Co}_2$	&		&\!\!\!\!\!$\Alt_5$, $\Alt_6$, $\Alt_8$	&\!\!\!\!\!\!\!\!\!\!$\typeC_2(3)$, $\typeC_3(2)$				&\!\!\!\!\!-- \\
$\Fi_{22}$			&		&\!\!\!\!\!--						&\!\!\!\!\!\!\!\!\!\!$\Atwo_3(3)$, $\Atwo_5(2)$, $\Dtwo_3(2)$	&\!\!\!\!\!--\\
\hline
$\Fi_{23}$			&6		&\!\!\!\!\!--						&\!\!\!\!\!\!\!\!\!\!$\Atwo_5(2)$, $\B_3(3)$, $\typeC_2(3)$, $\Dtwo_3(2)$   &\!\!\!\!\!$\Fi_{22}$ \\
$\mathrm{Co}_1$	&		&\!\!\!\!\!$\Alt_9$					&\!\!\!\!\!\!\!\!\!\!$\Atwo_3(3)$, $\typeC_2(3)$, $\D_4(2)$, $\G_2(4)$  & \!\!\!\!\!$\mathrm{M}_{12}$, $\mathrm{Suz}$ \\
$\mathrm{J}_4$	&		&\!\!\!\!\!--						&\!\!\!\!\!\!\!\!\!\!--											& \!\!\!\!\!$\mathrm{M}_{22}$\\
\hline
$\Fi_{24}'$		&7		&\!\!\!\!\!--						&\!\!\!\!\!\!\!\!\!\!$\Atwo_3(3)$, $\Atwo_4(2)$, $\Dtwo_3(2)$, $\D_4(3)$, $\G_2(3)$ 		& \!\!\!\!\!$\Fi_{22}$ \\
$\mathrm{B}$		&8		&\!\!\!\!\!--						&\!\!\!\!\!\!\!\!\!\!$\Dtwo_3(2)$, $\D_4(2)$, $\typeF_4(2)$, $\Etwo_6(2)$  		&\!\!\!\!\! $\mathrm{Co}_2$, $\Fi_{22}$ \\
$\mathrm{M}$		&10		&\!\!\!\!\!--						&\!\!\!\!\!\!\!\!\!\!--						& \!\!\!\!\!$\mathrm{B}$, $\mathrm{Co}_1$, $\Fi_{24}'$, $\mathrm{Suz}$, $\mathrm{Th}$   \\
\hline
\end{tabular} \caption{}
\label{spo2}
\end{table}

Every homomorphism from $\mcg{g}{b}$ with $g\geqslant 4$ to any of the alternating groups visible in the table is trivial, since by \cref{bgparis} the smallest alternating group to which $\mcg 4 b$ can map non-trivially has rank $2^3 \cdot (2^4-1) = 120$.

All the groups of Lie type occurring as simple factors are eliminated by \cref{lie type}, with the the exception of $\Atwo_5(2)$ inside $\Fi_{22}$; this group is eliminated by \cref{main thm reps}.

The last column of \cref{spo2} lists the remaining sporadic simple factors. Each of these is too small to allow for a non-trivial homomorphism from the relevant $\mcg{g(K)-1} {b+1}$.
\end{proof}

\newpage

\bibliographystyle{math}
\bibliography{bibliographymcg}

\bigskip
\noindent Dawid Kielak \hfill \href{mailto:dkielak@math.uni-bielefeld.de?subject=Smallest quotients of MCGs&body=Dear Dawid,}{\texttt{dkielak@math.uni-bielefeld.de}} \newline
\noindent Emilio Pierro \hfill  \href{mailto:e.pierro@mail.bbk.ac.uk}{\texttt{e.pierro@mail.bbk.ac.uk}} \newline
Fakult\"at f\"ur Mathematik  \newline
Universit\"at Bielefeld \newline
Postfach 100131  \newline
D-33501 Bielefeld \newline
Germany \newline

\end{document}